\documentclass[a4paper,12pt]{article}

\usepackage{amsfonts, amssymb, amscd,amsthm, amsmath, eucal}

\usepackage{amsbsy}
\usepackage[all]{xy}
\SelectTips{cm}{12}
\usepackage{hyperref}

\theoremstyle{plain} {
  \newtheorem{thm}{Theorem}[subsection]
  \newtheorem{cor}[thm]{Corollary}

  \newtheorem{clm}[thm]{Claim}

  \newtheorem{sssection}[thm]{}
}

\theoremstyle{definition}
{
  \newtheorem{defn}[thm]{Definition}
  \newtheorem{rem}[thm]{Remark}
  \newtheorem{example}[thm]{Example}
}
\renewcommand{\subsubsection}{\sssection\rm}

\newcommand{\id}{\mathrm{id}}
\newcommand{\Spec}{\mathrm{Spec}}
\newcommand{\Hom}{\mathrm{Hom}}

\newcommand{\Th}{\mathrm{Th}}

\newcommand{\Sm}{\mathcal Sm}

\newcommand{\Aff}{\mathbf {A}}
\newcommand{\Pro}{\mathbf {P}}
\newcommand{\Gr}{\mathrm{Gr}}
\newcommand{\F}{\mathbf {F}}

\let\minus=\smallsetminus

\newcommand{\SH}{\mathrm{SH}}
\newcommand{\cm}{\mathrm{cm}}
\newcommand \xra {\xrightarrow }

\newcommand \hra {\hookrightarrow }

\newcommand{\MSS}{\mathbf{MSS}}
\newcommand{\sSet}{\mathbf{sSet}}

\newcommand{\colim}{\operatornamewithlimits{colim}}

\begin{document}

\title
{A universality theorem for Voevodsky's algebraic cobordism spectrum}

\author{I.~Panin\footnote{Universit\"at Bielefeld, SFB 701, Bielefeld, Germany}
\footnote{Steklov Institute of Mathematics at St.~Petersburg, Russia}
\and K.~Pimenov\footnotemark[2]
\and O.~R{\"o}ndigs\footnote{Institut f\"ur Mathematik, Universit\"at Osnabr\"uck, Osnabr\"uck, Germany}}

\date{September 26, 2007\thanks{The authors thank
the SFB-701 at the Universit\"at Bielefeld, the RTN-Network HPRN-CT-2002-00287, the
RFFI-grant 03-01-00633a, and
INTAS-05-1000008-8118 for their support.}}

\maketitle

\begin{abstract}
An algebraic version of a theorem due to Quillen is proved. More precisely,
for a ground field $k$ we consider the motivic stable homotopy category
$\mathrm{SH}(k)$ of $\Pro^1$-spectra,
equipped with the symmetric monoidal structure described in
\cite{BGL}.
The algebraic cobordism $\Pro^1$-spectrum $\mathrm{MGL}$ is considered as a commutative
monoid equipped with a canonical orientation
$th^{\mathrm{MGL}} \in \mathrm{MGL}^{2,1}(\Th(\mathcal O(-1)))$.
For a commutative monoid $E$ in the category
$\mathrm{SH}(k)$
it is proved that assignment
$\varphi \mapsto \varphi(th^{\mathrm{MGL}})$
identifies the set of monoid homomorphisms
$\varphi\colon \mathrm{MGL} \to E$
in the motivic stable homotopy category
$\mathrm{SH}(k)$
with the set of all orientations of $E$.
This result was stated originally in a slightly different form by G.~Vezzosi in
\cite{Vez}.
\end{abstract}

\section{Introduction}

Quillen proved in \cite{Quillen:mu} that the formal group law associated to the complex cobordism
spectrum $\mathrm{MU}$
is the universal one on the Lazard ring. As a consequence, the set of orientations on a
commutative ring spectrum $E$ in the stable homotopy category is in bijective correspondence with the
set of homomorphisms of ring spectra from $\mathrm{MU}$ to $E$ in the stable homotopy category.
This result allowed a whole new approach to understanding the stable homotopy category, which
is still actively pursued today.

On the algebraic side of things, there is a similar $\Pro^1$-ring spectrum $\mathrm{MGL}$ in
the motivic stable homotopy category of a field $k$. The formal group law associated to 
$\mathrm{MGL}$ is not known to be the universal one, although unpublished work of Hopkins and
Morel claims this if $k$ has characteristic zero. Nevertheless, the set of orientations
on a $\Pro^1$-ring spectrum in the motivic stable homotopy category over $k$ can be identified
in the same fashion. 

\begin{thm}\label{thm:preliminaries}
  Let $E$ be a commutative $\Pro^1$-ring spectrum over $k$. The set of orientations
  on $E$ is in bijection with the set of homomorphisms of $\Pro^1$-ring spectra
  from $\mathrm{MGL}$ to $E$ in the motivic stable homotopy category over $k$.
\end{thm}

For a more detailed formulation, see~\ref{UniversalityThm}.
Our main motivation to write this paper was to prove the universality 
theorem~\ref{thm:preliminaries}
in a form convenient for its application in
\cite{CV}. Theorem~\ref{thm:preliminaries} was stated originally in a slightly different
form by G.~Vezzosi in \cite{Vez}, although he ignored certain aspects of 
the multiplicative structure on
$\mathrm{MGL}$.

\subsection{Preliminaries}

We refer to
\cite[Appendix]{BGL}
for the basic terminology, notation, constructions, definitions,
results. For the convenience of the reader we recall the
basic definitions. Let $S$ be a Noetherian scheme of
finite Krull dimension.
One may think of $S$ being the spectrum of a field or the integers.
Let $\Sm/S$ be the category of smooth
quasi-projective $S$-schemes, and let $\sSet$ be the category of
simplicial sets. A {\em motivic space over\/} $S$ is a functor
\[ A\colon \Sm/S^\mathrm{op} \to \sSet \]
(see \cite[A.1.1]{BGL}).
The category of motivic spaces over $S$ is denoted
$\mathbf{M}(S)$. This definition of a motivic space is different from the one considered
by Morel and Voevodsky in \cite{MV} -- they consider only those simplicial presheaves
which are sheaves in the Nisnevich topology on $\Sm/S$. With our definition the
Thomason-Trobaugh $K$-theory functor obtained by using big vector bundles is a motivic space
on the nose. It is not a simplicial Nisnevich sheaf. This is why we prefer to work with the
above notion of ``space''.

We write
$\mathrm{H}^{\cm}_\bullet (S)$
for the pointed motivic homotopy category
and
$\mathrm{SH}^\cm(S)$
for
the stable motivic homotopy category over $S$ as constructed in \cite[A.3.9, A.5.6]{BGL}.
By \cite[A.3.11 resp.~A.5.6]{BGL} there are canonical equivalences to
$\mathrm{H}_\bullet(S)$ of \cite{MV}
resp. $\mathrm{SH}(S)$ of \cite{V1}.
Both $\mathrm{H}^\cm_\bullet(S)$ and $\mathrm{SH}^\cm_\bullet(S)$  are equipped with
closed symmetric monoidal structures such that the $\Pro^1$-suspension spectrum functor
is a strict symmetric monoidal functor
\[ \Sigma^\infty_{\Pro^1} \colon \mathrm{H}^\cm_\bullet(S)\to \mathrm{SH}^\cm(S). \]
Here
$\Pro^1$
is considered as a motivic space pointed by
$\infty \in \Pro^1$.
The symmetric monoidal structure $(\wedge,\mathbb{I}_S = \Sigma^\infty_{\Pro^1}S_+)$
on the homotopy category
$\mathrm{SH}^\cm(S)$ is constructed on the model category level by
employing symmetric $\Pro^1$-spectra.
It satisfies the properties required by
Theorem 5.6 of Voevodsky congress talk
\cite{V1}. From now on we will usually omit the superscript $(-)^\cm$.

Every $\Pro^1$-spectrum $E = (E_0,E_1,\dotsc)$ represents a cohomology theory on the category
of pointed motivic spaces. Namely, for a pointed motivic space $(A,a)$ set
\[E^{p,q}(A,a)=\Hom_{\mathrm{SH}_\bullet (S)}(\Sigma^\infty_{\Pro^1}(A,a), \Sigma^{p,q}(E))\]
and
$E^{\ast,\ast}(A,a)= \oplus_{p,q} E^{p,q}(A,a)$.
This definition extends to motivic spaces via the functor $A\mapsto A_+$ which
adds a disjoint basepoint.
That is, for a non-pointed motivic space
$A$ set
$E^{p,q}(A)=E^{p,q}(A_+,+)$
and
$E^{\ast,\ast}(A)= \oplus_{p,q} E^{p,q}(A)$.
Recall that there is a canonical element in $E^{2n,n}(E_n)$, denoted as
$\Sigma^\infty_{\Pro^1}E_n(-n) \to E$. It is represented by
the canonical map $(\ast,\dotsc,\ast,E_n,E_n\wedge \Pro^1,\dotsc) \to (E_0,E_1,\dotsc,E_n,\dotsc)$
of $\Pro^1$-spectra.

Every
$X \in \Sm/S$
defines a representable motivic space constant in the simplicial direction
taking an $S$-smooth scheme $U$ to
$\mathrm{Hom}_{\Sm/S}(U,X)$. It is not possible in general to choose
a basepoint for representable motivic spaces. %This motivic space is non-pointed.
So we regard $S$-smooth varieties as motivic spaces (non-pointed)
and set
$$
E^{p,q}(X)=E^{p,q}(X_+,+).
$$

Given a $\Pro^1$-spectrum $E$ we will reduce the double grading on the cohomology
theory
$E^{\ast,\ast}$
to a grading. Namely, set
$E^m = \oplus_{m = p-2q}E^{p,q}$
and
$E^{\ast}=\oplus_{m}E^m$.
{\it We often write\/}
$E^{\ast}(k)$
{\it for}
$E^{\ast}(\mathrm{Spec}(k))$
{\it below}.

To complete this section, note that for us a
$\Pro^1$-{\em ring spectrum\/} is a monoid
$(E,\mu,e)$
in
$(\mathrm{SH(S)},\wedge, \mathbb{I}_S)$.
A {\em commutative\/} $\Pro^1$-ring
spectrum is a commutative monoid
$(E,\mu,e)$
in
$(\mathrm{SH(S)},\wedge, 1)$.
The cohomology theory $E^{\ast}$ defined by a
$\Pro^1$-ring spectrum is a ring cohomology theory.
The cohomology theory $E^{\ast}$ defined by a commutative
$\Pro^1$-ring spectrum is a ring cohomology theory,
however it is not necessary graded commutative.
The cohomology theory $E^{\ast}$ defined by an oriented commutative
$\Pro^1$-ring spectrum is a graded commutative ring cohomology theory
\cite{PY}.

\subsection{Oriented commutative ring spectra}
Following Adams and Morel we define an orientation of a commutative
$\Pro^1$-ring spectrum. However we prefer to use Thom classes instead of
Chern classes. Consider the pointed motivic space
$\Pro^{\infty}= \colim_{n\geq 0} \Pro^n$ having base point
$g_1 \colon S = \Pro^0 \hra \Pro^\infty$.

The tautological ``vector bundle''
$\mathcal{T}(1) = \mathcal{O}_{\Pro^\infty}(-1)$ is also known as the
Hopf bundle. It has zero section $z\colon \Pro^\infty \hra \mathcal{T}(1)$.
The fiber over the point
$g_1 \in \Pro^{\infty}$
is
$\mathbb{A}^1$.
For a vector bundle $V$ over a smooth
$S$-scheme $X$, with zero section $z\colon X \hra V$,
its {\em Thom space\/} $\Th(V)$
is
the Nisnevich sheaf associated to
the presheaf
$Y \mapsto V(Y)/\bigl(V\minus z(X)\bigr)(Y)$
on the Nisnevich site
$\Sm/S$.
%The  of $V$ is defined as that
%Nisnevich sheaf regarded as a presheaf.
In particular, $\Th(V)$ is a pointed motivic space in the
sense of
\cite[Defn.~A.1.1]{BGL}.
It coincides with Voevodsky's Thom space \cite[p. 422]{V1},
since
$\Th(V)$
already is a Nisnevich sheaf.
The Thom space of the Hopf bundle is then defined as the colimit
$\Th(\mathcal{T}(1)) = \colim_{n\geq 0} \Th\bigl(\mathcal{O}_{\Pro^n}(-1)\bigr)$.
Abbreviate $T =\Th(\Aff^1_S)$.

Let $E$ be a commutative $\Pro^1$-ring spectrum. The unit gives rise
to an element
$1\in E^{0,0}(\Spec(k)_+)$.
Applying the
$\Pro^1$-suspension
isomorphism to that element we get an element
$\Sigma_{\Pro^1}(1) \in E^{2,1}(\Pro^1,\infty)$.
The canonical covering of $\Pro^1$ defines
motivic weak equivalences
\[ \xymatrix{\Pro^1 \ar[r]^-\sim & \Pro^1/\Aff^1 &
\Aff^1/\Aff^1\minus \lbrace 0 \rbrace = T \ar[l]_-\sim} \]
of pointed motivic spaces inducing isomorphisms
$E(\Pro^1,\infty) \leftarrow E(\Aff^1/\Aff^1\minus \lbrace 0 \rbrace) \rightarrow  E(T)$.
Let
$\Sigma_{T}(1)$
be the image of
$\Sigma_{\Pro^1}(1)$
in
$E^{2,1}(T)$.

\begin{defn}
\label{OrientationViaThom}
Let $E$ be a commutative $\Pro^1$-ring spectrum.
A {\em Thom orientation of\/}
$E$ is an element
$th \in E^{2,1}(\Th(\mathcal T(1))$
%= =E^{2,1}(\mathbb{MGL}_1)$ %{\Pro^{\infty}}(\mathcal T(1))$
such that its restriction to the Thom space of the fibre over the distinguished
point coincides with the element
$\Sigma_{T}(1) \in E^{2,1}(T)$.
A {\em Chern orientation of\/}
$E$ is an element
$c \in E^{2,1}(\Pro^{\infty})$
such that
$c|_{\Pro^1}= - \Sigma_{\Pro^1}(1)$.
{\em An orientation} of $E$ is either a Thom orientation or a Chern orientation.
%Two Thom orientations of $E$ coincide if respecting Thom elements coincides.
%Two Chern orientations of $E$ coincide if respecting Chern elements coincides.
One says that a Thom orientation $th$ of $E$ coincides
with a Chern orientation $c$ of $E$ provided that
$c = z^*(th)$ or equivalently
the element
$th$ coincides with the one
$th(\mathcal O(-1))$
given by
(\ref{ThomClass})
below.
\end{defn}

% It should be regarded
% as the Thom class of the Hopf bundle $\mathcal{O}(-1)$ on $\Pro^\infty$.
%%% \begin{defn}
%%% \label{OrientationViaChern}
%%% Let $E$ be a commutative ring $\Pro^1$-spectrum.
%%% \end{defn}

\begin{rem}
\label{ThomAndChern}
The element $th$ should be regarded as the Thom class of the tautological
line bundle
$\mathcal{T}(1)= \mathcal O(-1)$
over
$\Pro^{\infty}$.
The element $c$ should be regarded as the Chern class of the tautological
line bundle
$\mathcal{T}(1)= \mathcal O(-1)$
over
$\Pro^{\infty}$.
\end{rem}

%\begin{rem}
%\label{ThomAndChern}
%Let $c$ be an orientation of $E$. Consider
%$th(\mathcal T(1)) \in E^{2,1}_{\Pro^{\infty}}(\mathcal T(1))$
%given by
%(\ref{ThomClass})
%and set \ $th= th(\mathcal T(1))$.
%It is straight forward to check that
%$th|_{Th(1)}= \Sigma_{\Pro^1}(1)$.
%Thus $th$ is an orientation of $E$.
%Clearly
%$c =z^*(th) \in E^{2,1}(\Pro^{\infty})$.
%Whence the two definition of orientations of $E$
%are equivalent.
%\end{rem}

\begin{example}
\label{OrientationsOfMGLandK}
The following orientations given right below are relevant for our work.
Here $\mathrm{MGL}$ denotes the $\Pro^1$-ring spectrum
representing algebraic cobordism obtained below in Definition~\ref{textMGL},
and $\mathrm{BGL}$ denotes the $\Pro^1$-ring spectrum
representing algebraic $K$-theory
constructed in
\cite[Theorem 2.2.1]{BGL}.
\begin{itemize}
\item
Let
$u_1: \Sigma^{\infty}_{\Pro^1}(\Th(\mathcal{T}(1)))(-1) \to \mathrm{MGL}$
be the canonical map of
$\Pro^1$-spectra. Set
$th^{\mathrm{MGL}} =u_1 \in \mathrm{MGL}^{2,1}(\Th(\mathcal{T}(1)))$.
Since
$th^{\mathrm{MGL}}|_{\Th(\mathbf{1})}= \Sigma_{\Pro^1}(1)$
in
$\mathrm{MGL}^{2,1}(\Th(\mathbf{1}))$,
the class
$th^{\mathrm{MGL}}$
is an orientation of
$\mathrm{MGL}$.
\item
Set
$c =(- \beta) \cup ([\mathcal O]-[\mathcal O(1)])
\in \mathrm{BGL}^{2,1}(\Pro^{\infty})$.
The relation
(11)
from
\cite{BGL}
shows that
the class $c$ is an orientation of
$\mathrm{BGL}$.
%Consider
%$th(\mathcal O(-1)) \in \mathrm{BGL}^{2,1}_{\Pro^{\infty}}(\mathcal O(-1))$
%given by
%(\ref{ThomClass})
%and set
%$th^K = th(\mathcal O(-1))$.
%The class
%$th^K$
%is the same orientation of
%$\mathrm{BGL}$.
\end{itemize}
\end{example}

\section{Oriented ring spectra and infinite Grassmannians}
\label{OrientedSpectraAndTheory}
Let $(E,c)$ be an oriented commutative $\Pro^1$-ring spectrum. In this section we compute
the $E$-cohomology of infinite
Grassmannians and their products. The results
are the expected ones -- see Theorems
\ref{CohomologyOfGr} and \ref{CohomologyOfGrGr}.

The oriented $\Pro^1$-ring spectrum $(E,c)$ defines an oriented cohomology theory on
$\Sm/S$ in the sense of
\cite[Defn.~3.1]{PSorcoh}
as follows.
The restriction of the functor $E^{\ast,\ast}$ to the category
$\Sm/S$ is a ring cohomology theory.
By
\cite[Th.~3.35]{PSorcoh}
it remains to construct
a Chern structure on
$E^{\ast,\ast}|_{\Sm/S}$
in the sense of
\cite[Defn.~3.2]{PSorcoh}.
Let
$\mathrm{H}_\bullet(k)$
be the homotopy category of pointed motivic spaces over $k$.
The functor isomorphism
$\Hom_{\mathrm{H}_\bullet(k)}(- , \Pro^{\infty}) \to \mathrm{Pic}(-)$
%$\Hom_{\mathrm{H}_\bullet(S)}(- , \Pro^{\infty}) \to \mathrm{Pic}(-)$
on the category
$\Sm/S$
provided by
\cite[Thm.~4.3.8]{MV}
sends the class of the identity map
$\Pro^\infty \to {\Pro^{\infty}}$ to the class of the tautological line bundle
$\mathcal{O}(-1)$
over
$\Pro^{\infty}$.
For a line bundle $L$ over
$X\in \Sm/S$ let
$[L]$ be the class of $L$ in the group
$\mathrm{Pic}(X)$.
Let
$f_L\colon X \to \Pro^{\infty}$
be a morphism in
$\mathrm{H(k)}$ corresponding to
the class $[L]$
under the functor isomorphism above.
For a line bundle $L$ over $X\in \Sm/S$
set
$c(L)=f^\ast_L(c) \in E^{2,1}(X)$.
Clearly,
$c(\mathcal{O}(-1))=c$.
The assignment
$L/X \mapsto c(L)$
is a Chern structure on
$E^{\ast,\ast}|_{\Sm/S}$
since
$c|_{\Pro^1}= - \Sigma_{\Pro^1}(1) \in E^{2,1}(\Pro^1,\infty)$.
With that Chern structure
$E^{\ast,\ast}|_{\Sm/S}$
is an oriented ring cohomology theory
in the sense of
\cite{PSorcoh}.
In particular,
$(\mathrm{BGL},c^K)$
defines an oriented ring cohomology theory on
$\Sm/S$.

Given this Chern structure, one
obtains a theory of Thom classes
$V/X \mapsto th(V) \in E^{2\mathrm{rank}(V),\mathrm{rank}(V)}\bigl(\Th_X(V)\bigr)$
on the cohomology theory
$E^{\ast,\ast}|_{\Sm/S}$
in the sense of
\cite[Defn.~3.32]{PSorcoh} as follows.
There is a unique theory of Chern classes
$V \mapsto c_i(V) \in E^{2i,i}(X)$
such that for every line bundle $L$ on $X$ one has
$c_1(L)=c(L)$. For a rank $r$ vector bundle
$V$ over $X$ consider the vector bundle
$W:= {\mathbf{1}} \oplus V$
and the associated projective vector bundle
$\Pro(W)$
of lines in $W$.
Set
\begin{eqnarray}
\label{ThomBarClass}
\bar th(V)= c_r(p^\ast(V) \otimes \mathcal{O}_{\Pro(W)}(1)) \in E^{2r,r}(\Pro(W)).
\end{eqnarray}
It follows from
\cite[Cor.~3.18]{PSorcoh}
that the support extension map
\[E^{2r,r}\bigl(\Pro(W)/(\Pro(W)\smallsetminus \Pro(\mathbf{1}))\bigr)
\to E^{2r,r}\bigl(\Pro(W)\bigr)\]
is injective and
$\bar th(E) \in E^{2r,r}\bigl(\Pro(W)/(\Pro(W)\smallsetminus \Pro(\mathbf{1}))\bigr) $.
Set
\begin{eqnarray}
\label{ThomClass}
th(E)= j^\ast(\bar th(E)) \in E^{2r,r}\bigl(\Th_X(V)\bigr),
\end{eqnarray}
where
$j\colon \Th_X(V) \to \Pro(W)/ (\Pro(W) \smallsetminus \Pro({\bf 1}))$
is the canonical motivic weak equivalence of pointed motivic spaces
induced by the open embedding $V\hra \Pro(W)$.
The assignment $V/X$ to $th(V)$ is a theory of Thom classes
on $E^{\ast,\ast}|_{\Sm/S}$
(see the proof of
\cite[Thm.~3.35]{PSorcoh}). Hence the Thom classes are natural,
multiplicative and satisfy the following Thom isomorphism property.

\begin{thm}
\label{ThomIsomorphism}
For a rank $r$ vector bundle
$p\colon V \to X$
on $X\in \Sm/S$ with zero section $z\colon X\hra V$, the map
\[ -\cup th(V)\colon  E^{\ast,\ast}(X) \to E^{\ast+2r,\ast+r}\bigl(V/(V\minus z(X))\bigr) \]
is an isomorphism of two-sided
$E^{\ast,\ast}(X)$-modules,
where
$-\cup th(V)$
is written for the composition map
$\bigl(-\cup th(V)\bigr) \circ p^\ast$.
\end{thm}

\begin{proof}
  See \cite[Defn.~3.32.(4)]{PSorcoh}.
\end{proof}

%\begin{rem}
%\label{ThomViaChern}
%If an element $th$ is an orientation of $E$ then set
%$c:=z^*(th) \in E^{2,1}(\Pro^{\infty})$
%as in Remark
%\ref{ThomAndChern}.
%Then $(E,c)$ is an oriented commutative ring
%$\Pro^1$-spectrum and
%the Thom class
%$th(\mathcal O(-1)) \in E^{2,1}_{\Pro^{\infty}}(\mathcal O(-1))$
%given by
%(\ref{ThomClass})
%coincides with the element $th$
%(see
%\cite[Thm.??]{PSorcoh}).
%In particular,
%$th(\mathcal O(-1))|_{Th({\bf 1})}= \Sigma_{\Pro^1}(1)$.
%\end{rem}

Analogous to
\cite[p. 422]{V1}
one obtains for
vector bundles $V \to X$ and $W\to Y$ in $\Sm/S$
a canonical map of pointed motivic spaces
$\Th(V) \wedge \Th(W) \to \Th(V \times_S W)$
which is a motivic weak equivalence as defined in \cite[Defn.~3.1.6]{BGL}.
In fact, the canonical map becomes an isomorphism
after Nisnevich (even Zariski) sheafification.
Taking
$Y=S$
and
$W= \mathbf{1}$ the trivial line bundle
yields a motivic weak equivalence
$\Th(V) \wedge T \to Th(V \oplus \mathbf{1})$.
The canonical covering of $\Pro^1$ defines
motivic weak equivalences
\[ \xymatrix{T = \Aff^1/\Aff^1\minus \lbrace 0 \rbrace \ar[r]^-\sim & \Pro^1/\Aff^1 &
 \Pro^1 \ar[l]_-\sim} \]
and the arrow
$T = \Aff^1/\Aff^1\minus \lbrace 0 \rbrace \to  \Pro^1/\Pro^1\minus \lbrace 0 \rbrace$
is an isomorphism.
Hence one may switch between $T$ and $\Pro^1$ as desired.

\begin{cor}
\label{ThomAndSuspension}
For
$W=V \oplus {\mathbf{1}}$
consider the  motivic weak equivalences
%%% $\epsilon\colon \Th(V) \wedge \Pro^1 \to \Th(W)$
%%% be the composite motivic weak equivalence
\[\epsilon\colon \Th(V) \wedge \Pro^1 \to \Th(V) \wedge \Pro^1/\Aff^1 \leftarrow \Th(V) \wedge T \to \Th(W)\]
of pointed motivic spaces over $S$.
The diagram
\[
\xymatrix{
E^{\ast+2r,\ast+r}(\Th(V)) \ar[r]^-{\Sigma_{\Pro^1}} & E^{\ast+2r+2,\ast+r+1}(\Th(V) \wedge \Pro^1) \\
E^{\ast+2r,\ast+r}(\Th(V)) \ar[r]^-{\Sigma_{T}} \ar[u]^-{\mathrm{id}} & E^{\ast+2r+2,\ast+r+1}(\Th(W))
 \ar[u]_-{\epsilon^\ast}\\
E^{\ast,\ast}(X)  \ar[r]^-{\mathrm{id}} \ar[u]^-{-\cup th(V)}   &  E^{\ast,\ast}(X) \ar[u]_-{-\cup th(W)}}
\]
commutes.
\end{cor}

Let $\Gr(n,n+m)$ be the Grassmann scheme of $n$-dimensional linear
subspaces of $\Aff^{n+m}_S$. The closed embedding $\Aff^{n+m} = \Aff^{n+m}\times \{0\}
\hra \Aff^{n+m+1}$ defines a closed embedding
\begin{equation}\label{eq:1}
  \Gr(n,n+m)\hra \Gr(n,n+m+1).
\end{equation}
%In particular, the Grassmann scheme may be viewed as a motivic
%space pointed by $S = \Gr(n,n) \hra \Gr(n,n+m)$.
The tautological vector bundle
is denoted $\mathcal{T}(n,n+m)\to \Gr(n,n+m)$. The closed embedding~(\ref{eq:1})
is covered by a map of vector bundles $\mathcal{T}(n,n+m)\hra \mathcal{T}(n,n+m+1)$. Let
$\Gr(n) = \colim_{m\geq 0} \Gr(n,n+m)$, $\mathcal{T}(n) = \colim_{m\geq 0} \mathcal{T}(n,n+m)$
and $\Th(\mathcal{T}(n)) = \colim_{m\geq 0} \Th(\mathcal{T}(n,n+m))$.
These colimits are taken in the category of motivic spaces over $S$.

\begin{rem}
\label{FiniteGrassmannians}
It is not difficult to prove that
$E^{\ast,\ast}(\Gr(n,n+m))$
is multiplicatively generated by the Chern classes
$c_i(\mathcal{T}(n,n+m))$
of the vector bundle
$\mathcal{T}(n,n+m)$. This proves the surjectivity of the pull-back maps
$E^{\ast,\ast}(\Gr(n,n+m+1)) \to E^{\ast,\ast}(\Gr(n,n+m))$
and shows that the canonical map
$E^{\ast,\ast}(\Gr(n)) \to {\varprojlim}E^{\ast,\ast}(\Gr(n,n+m))$
is an isomorphism. Thus for each $i$ there exists a unique element
$c_i= c_i(\mathcal T(n)) \in E^{2i,i}(\Gr(n))$
which for each $m$ restricts to the element
$c_i(\mathcal{T}(n,n+m))$ under the obvious pull-back map.
\end{rem}

\begin{thm}
\label{CohomologyOfGr}
Let $E$ be an oriented $\Pro^1$-ring spectrum.
Then
\[
E^{\ast,\ast}(\Gr(n))= E^{\ast,\ast}(k)[[c_1,c_2, \dots, c_n]]
\]
is the formal power series ring.
%where
%$c_i:= c_i(\mathcal T(n)) \in E^{2i,i}(\Gr(n))$
%denotes the $i$-th Chern class of the tautological bundle
%$\mathcal T(n)$.
The inclusion
$\mathrm{inc}_n\colon \Gr(n) \hra \Gr(n+1)$
satisfies
$\mathrm{inc}_n^\ast(c_m)=c_m$
for $m < n+1$ and
$\mathrm{inc}^\ast_n(c_{n+1})=0$.
\end{thm}

\begin{proof}
The case $n=1$ is well-known
(see for instance
\cite[Thm.~3.9]{PSorcoh}).
For a
finite dimensional vector space $W$ and a positive integer $m$ let
$\F(m,W)$ be the flag variety of flags
$W_1 \subset W_2 \subset \dots \subset W_m$
of linear subspaces of $W$ such that
the dimension of $W_i$ is $i$.
Let
$\mathcal T^i(m,W)$
be the tautological rank $i$ vector bundle on
$\F(m,W)$.

Let $V =\Aff^\infty$ be an infinite dimensional vector bundle over $S$ and
set $e=(1,0,\dotsc)$. Then $V_n$ denotes the $n$-fold product of $V$, and
$e^n_i\in V_n$ the vector $(0,\dotsc,0,e,0,\dotsc,0)$ having $e$ precisely at the
$i$th position.
Let
$F(m)=\colim_W \F(m,W)$
and let
$\mathcal T^i(m)=\colim_W \mathcal T^i(m,W)$,
where $W$
runs over all finite-dimensional vector subspaces of $V_n$.
Thus we have a flag
$\mathcal T^1(m) \subset \mathcal T^2(m) \subset \dots \subset \mathcal T^m(m)$
of vector bundles over
$F(m)$.
Set
$L^i(m)=\mathcal T^i(m)/\mathcal T^{i-1}(m)$.
It is a line bundle over
$F(m)$.

Consider the morphism
$p_m\colon F(m) \to F(m-1)$
which takes a flag
$W_1 \subset W_2 \subset \dots \subset W_m$
to the flag
$W_1 \subset W_2 \subset \dots \subset W_{m-1}$.
It is a projective vector bundle over
$F(m-1)$ such that the line bundle
$L^i(m)$ is its tautological line bundle.
Thus there exists a tower of projective vector bundles
$F(m) \to F(m-1) \to \dots \to F(1)= \Pro(V_n)$.
The projective bundle theorem implies that
\[
E^{\ast,\ast}(F(n))=E^{\ast,\ast}(k)[[t_1,t_2, \dots, t_n]]
\]
(the formal power series in $n$ variables),
where
$t_i=c(L^i(n))$
is the first Chern class of the line bundle
$L^i(n)$
over
$F(n)$.

Consider the morphism
$q\colon F(n) \to \Gr(n)$,
which takes a flag
$W_1 \subset W_2 \subset \dots \subset W_n$
to the space
$W_n$.
It can be decomposed as a tower of projective vector
bundles. In particular,  the pull-back map
$q^\ast\colon E^{\ast,\ast}(\Gr(n)) \to E^{\ast,\ast}(F(n))$
is a monomorphism.
It takes the class
$c_i$
to the symmetric polynomial
$\sigma_i= t_1t_2 \dots t_i + \dots + t_{n-i+1}\dots t_{n-1}t_n$.
So the image of
$q^\ast$
contains
$E^{\ast,\ast}(k)[[\sigma_1,\sigma_2, \dots, \sigma_n]]$.
It remains to check that the image of $q^\ast$ is contained in
$E^{\ast,\ast}(k)[[\sigma_1,\sigma_2, \dots, \sigma_n]]$.
To do that consider another variety.

Namely, let
$V^0$
be the $n$-dimensional subspace of
$V_n$
generated by the vectors
$e^n_i$'s.
Let
$l^n_i$
be the line generated by the vector
$e^n_i$.
Let
$V^0_i$
be a subspace of $V^0$
generated by all
$e^n_j$'s
with
$j \leq i$.
So one has a flag
$V^0_1 \subset V^0_2 \subset \dots \subset V^0_n$.
We denote this flag
$F^0$.
For each vector subspace
$W$ in $V_n$ containing $V^0$
consider three algebraic subgroups of the general linear group
$\mathbb{GL}_W$. Namely, set
$$
P_W= Stab(V^0),\ B_W= Stab(F^0), \ T_W= Stab(l^n_1,l^n_2,\dots,l^n_n).
$$
The group
$T_W$
stabilizes
each line
$l^n_i$.
Clearly,
$T_W \subset B_W \subset P_W$
and
$\Gr(n,W)= \mathbb{GL}_W/P_W$,
$\F(n,W)= \mathbb{GL}_W/B_W$
Set
$M(n,W)= \mathbb{GL}_W/T_W$.
One has a tower of obvious morphisms
$$
M(n,W) \xra{r_W} \F(n,W) \xra{q_W} \Gr(n,W).
$$
Set
$M(n)= \colim_W M(n,W)$,
where $W$ runs over all finite dimensional subspace
$W$ of $V_n$ containing $V^0$. Now one has a tower
of morphisms
$$
M(n) \xra{r} F(n) \xra{q} \Gr(n).
$$
The morphisms $r_W$ can be decomposed in a tower
of affine bundles. Hence it induces an isomorphism
on any cohomology theory. The same then holds
for the morphism $r$ and
$$
E^{\ast,\ast}(M(n))=E^{\ast,\ast}(k)[[t_1,t_2, \dots, t_n]].
$$
Permuting vectors
$e^n_i$'s
yields an inclusion
$\Sigma_n \subset GL(V^0)$
of the symmetric group
$\Sigma_n$
in
$\mathbb{GL}(V^0)$.
The action of
$\Sigma_n$
by the conjugation on
$\mathbb{GL}_W$
normalizes the subgroups
$T_W$ and $P_W$.
Thus
$\Sigma_n$
acts as on
$M(n)$ so on $\Gr(n)$
and the morphism
$q \circ r: M(n) \to \Gr(n)$
respects this action.
Note that the action of
$\Sigma_n$
on
$\Gr(n)$
is trivial
and the action of
$\Sigma_n$
on
$E^{\ast,\ast}(M(n))$
permutes the variable
$t_1, t_2, \dots, t_n$.
Thus the image of
$(q \circ r)^*$
is contained in
$E^{\ast,\ast}(k)[[\sigma_1,\sigma_2, \dots, \sigma_n]]$.
Whence the same holds for
the image of
$q^*$.
The Theorem is proven.
\end{proof}

The projection
from the product
$\Gr(m) \times \Gr(n)$,
to the $j$-th factor is called $p_j$.
%and the vector bundles
%$p^*_1(\mathcal T(m))$
%and
%$p^*_2(\mathcal T(n))$
%over
%$G(m) \times \Gr(n)$.
For every integer
$i \geq 0$
set
$c^{\prime}_i=p^*_1(c_i(\mathcal T(m)))$
and
$c^{\prime\prime}_i=p^*_2(c_i(\mathcal T(n)))$

\begin{thm}
\label{CohomologyOfGrGr}
Suppose $E$ is an oriented commutative $\Pro^1$-ring spectrum.
There is an isomorphism
\[
E^{\ast,\ast}\bigl((\Gr(m)\times \Gr(n))\bigr)= E^{\ast,\ast}(k)[[c^{\prime}_1,c^{\prime}_2, \dots, c^{\prime}_m,
c^{\prime\prime}_1,c^{\prime\prime}_2, \dots, c^{\prime\prime}_n]]
\]
is the formal power series on the
$c^{\prime}_i$'s
and
$c^{\prime\prime}_j$'s.
The inclusion
$i_{m,n}\colon G(m) \times \Gr(n) \hra G(m+1) \times G(n+1)$
satisfies
$i^*_{m,n}(c^{\prime}_r)=c^{\prime}_r$
for $r < m+1$,
$i^*_{m,n}(c^{\prime}_{m+1})=0$,
and
$i^*_{m,n}(c^{\prime\prime}_r)=c^{\prime\prime}_r$
for $r < n+1$,
$i^*_{m,n}(c^{\prime\prime}_{n+1})=0$.
\end{thm}

\begin{proof}
Follows as in the proof of Theorem~\ref{CohomologyOfGr}.
\end{proof}

\subsection{The symmetric ring spectrum representing algebraic cobordism}
\label{SpectrumMGL}

To give a construction of the
symmetric
$\Pro^1$-ring spectrum
$\mathrm{MGL}$,
recall the external product of Thom spaces
%For a vector bundle $V$ over a smooth
%$S$-scheme $X$ with zero section $z\colon X \hra V$
%let the {\em Thom space\/} $\Th(V)$ of $V$ be the Nisnevich sheaf associated to
%the presheaf
%$Y \mapsto V(Y)/\bigl(V\minus z(X)\bigr)(Y)$
%on the Nisnevich site
%$\Sm/S$.
%Since sheaves are presheaves, $\Th(V)$ is a pointed motivic space in the
%sense of
%\cite[Defn.~A.1.1]{BGL} which
%coincides with Voevodsky's Thom space \cite[p. 422]{V1}.
described in
\cite[p.~422]{V1}.
For
vector bundles $V \to X$ and $W\to Y$ in $\Sm/S$
one obtains a canonical map of pointed motivic spaces
$\Th(V) \wedge \Th(W) \to \Th(V \times_S W)$
which is a motivic weak equivalence as defined in \cite[Defn.~3.1.6]{BGL}.
In fact, the canonical map becomes an isomorphism
after Nisnevich (even Zariski) sheafification.

%Firstly consider a Nisnevich sheaf
%$T$ on $\Sm/S$
%associated with the presheaf
%$X \mapsto \Aff^1(X)/(\Aff^1 - \{0\})(X)$.
%The $T$ regarded as a presheaf is a motivic space in the
%sence of
%\cite[Defn.3.1.1]{BGL}.
%It is the Thom space
%$Th(\bf{1})$
%of the trivial rank one vector bundle
%$\bf{1}$ over $S$.
%Recall
%\cite[p. 422]{V1}
%that for
%$X,Y \in \Sm/S$
%and vector bundles
%$E$ over $X$
%and
%$F$ over $Y$
%and
%$E \times F$ over $X \times Y$
%there is a canonical isomorphism of spaces
%$Th(E) \wedge Th(F) = Th(E \times F)$.
%Taking $Y=S$ and $F= \bf{1}$ we get a canonical isomorphism
%of spaces
%$Th(E) \wedge T = Th(E \oplus \bf{1})$.
%Define the pointed motivic space
%$T$ as
%the Thom space
%$\Th(\mathbf{1})$
%of the trivial rank one vector bundle
%$\mathbf{1}$ over $S$.
The algebraic cobordism spectrum appears naturally as a $T$-spectrum, not as a
$\Pro^1$-spectrum. Hence we describe it as a symmetric $T$-ring spectrum and obtain
a symmetric $\Pro^1$-ring spectrum (and in particular a $\Pro^1$-ring spectrum) by
switching the suspension coordinate (see \cite[A.6.9]{BGL}).
%Now we are ready to recall a construction of the symmetric ring
%$T$-spectrum
%$\mathbb{MGL}$.
%with
%$T= \Aff^1/(\Aff^1 - \{0\})$
%After that we explain how to make it a symmetric ring
%$\Pro^1$-spectrum.
For $m, n \geq 0$ let $\mathcal{T}(n,mn)\to \Gr(n,mn)$ denote the tautological
vector bundle over the Grassmann scheme of $n$-dimensional linear subspaces of
$\Aff^{mn}_S = \Aff^m_S\times_S \dotsm \times_S \Aff^m_S$. Permuting the copies
of $\Aff^m_S$ induces a $\Sigma_n$-action on $\mathcal{T}(n,mn)$ and $\Gr(n,mn)$
such that the bundle projection is equivariant.
The closed embedding $\Aff^m_S = \Aff^m_S\times \lbrace 0 \rbrace
\hra \Aff^{m+1}_S$ defines a closed $\Sigma_n$-equivariant
embedding $\Gr(n,mn)\hra \Gr(n,(m+1)n)$.
In particular, $\Gr(n,mn)$ is pointed by $g_n\colon S = \Gr(n,n)\hra \Gr(n,mn)$.
The fiber of $\Gr(n,mn)$ over $g_n$ is $\Aff^n_S$.
Let $\Gr(n)$ be the colimit
of the sequence
\[ \Gr(n,n)\hra \Gr(n,2n) \hra \dotsm \hra \Gr(n,mn)\hra \dotsm \]
in the category of pointed motivic spaces over $S$. The pullback diagram
\[ \xymatrix{ \mathcal{T}(n,mn) \ar[r] \ar[d] & \mathcal{T}(n,(m+1)n) \ar[d] \\
                       \Gr(n,mn) \ar[r] & \Gr(n,(m+1)n)} \]
induces a $\Sigma_n$-equivariant
inclusion of Thom spaces \[\Th(\mathcal{T}(n,mn)) \hra \Th(\mathcal{T}(n,(m+1)n)).\]
Let $\mathbb{MGL}_n$ denote the colimit of the resulting sequence
\begin{eqnarray}\label{eq:mglt}
 \mathbb{MGL}_n = \colim_{m\geq n} \Th(\mathcal{T}(n,mn))
\end{eqnarray}
with the induced $\Sigma_n$-action.
There is a closed embedding
\begin{eqnarray}\label{eq:111}
	\Gr(n,mn)\times \Gr(p,mp) \hra \Gr(n+p,m(n+p))
\end{eqnarray}
which sends the linear subspaces
$V\hra \Aff^{mn}$ and $W\hra \Aff^{mp}$ to the product subspace
$V\times W \hra \Aff^{mn}\times \Aff^{mp} = \Aff^{m(n+p)}$.
In particular
$(g_n,g_p)$ maps to $g_{n+p}$. The inclusion~(\ref{eq:111}) is covered by a map of
tautological vector bundles and thus gives a canonical map of Thom spaces
\begin{eqnarray}\label{eq:2}
	\Th(\mathcal{T}(n,mn)) \wedge \Th(\mathcal{T}(p,mp)) \to \Th(\mathcal{T}(n+p,
	  m(n+p)))
\end{eqnarray}
which is compatible with the colimit~(\ref{eq:mglt}). Furthermore, the map~(\ref{eq:2})
is $\Sigma_n\times \Sigma_p$-equivariant, where the product acts on the target
via the standard inclusion $\Sigma_n\times \Sigma_p\subseteq \Sigma_{n+p}$.
After taking colimits, the result is a $\Sigma_n\times \Sigma_p$-equivariant map
\begin{eqnarray}
\label{MglProduct} \mu_{n,p}\colon \mathbb{MGL}_n\wedge \mathbb{MGL}_p \to
\mathbb{MGL}_{n+p}
\end{eqnarray}
of pointed motivic spaces (see \cite[p.~422]{V1}).
The inclusion of the fiber $\Aff^p$ over $g_p$ in $\mathcal{T}(p)$ induces
an inclusion $\Th(\mathbb{A}^p)\subset \Th(\mathcal T(p))=\mathbb{MGL}_p$.
Precomposing it with the canonical $\Sigma_p$-equivariant map of pointed motivic spaces
\[
\Th(\mathbb{A}^1)\wedge \Th(\mathbb{A}^1)\wedge \dotsm \wedge \Th(\mathbb{A}^1) \to
\Th(\mathbb{A}^p)
\]
defines a family of maps $e_p\colon(\Sigma^\infty_T S_+)_p = T^{\wedge p} \to \mathbb{MGL}_p$.
Inserting it in the inclusion (\ref{MglProduct}) yields
$\Sigma_n\times \Sigma_p$-equivariant structure maps
\begin{eqnarray}\label{eq:str-mglt}
\label{MglSuspension} \mathbb{MGL}_n \wedge \Th(\mathbb{A}^1)\wedge
\Th(\mathbb{A}^1)\wedge \dotsm \wedge \Th(\mathbb{A}^1)
\to
\mathbb{MGL}_{n+p}
\end{eqnarray}
of the symmetric $T$-spectrum $\mathbb{MGL}$.
The family of $\Sigma_n\times \Sigma_p$-equivariant maps~(\ref{MglProduct})
form a commutative, associative and unital multiplication
on the symmetric $T$-spectrum $\mathbb{MGL}$ (see
\cite[Sect.~4.3]{J}). Regarded as a $T$-spectrum it coincides with
Voevodsky's spectrum $\mathbf{MGL}$ described in \cite[6.3]{V1}.

%The space
%$\bar T$
%is defined just above Corollary
%\ref{ThomIsomorphism}.
%be a Nisnevich sheaf associated with the presheaf
%$X \mapsto \Pro^1(X)/(\Pro^1 - \{0\})(X)$
%on the Nisnevich site
%$\Sm/S$.
%The $\bar T$ regarded as a presheaf is a motivic space in the
%sence of
%\cite[Defn.3.1.1]{BGL}.

Let $\overline{T}$ be the Nisnevich sheaf associated to the presheaf
$X \mapsto \Pro^1(X)/(\Pro^1 - \{0\})(X)$
on the Nisnevich site
$\Sm/S$. The canonical covering of $\Pro^1$ supplies an isomorphism
\[ \xymatrix{ T =\Th(\Aff^1_S) \ar[r]^-\cong &\overline{T}} \]
of pointed motivic spaces. This isomorphism induces an isomorphism
$\MSS_T(S) \cong \MSS_{\overline{T}}(S)$ of the categories of
symmetric $T$-spectra and symmetric $\overline{T}$-spectra.
In particular, $\mathbb{MGL}$ may be regarded as a symmetric $\overline{T}$-spectrum
by just changing the structure maps up to an isomorphism.
Note that the isomorphism of categories respects both the
symmetric monoidal structure and the model structure.
The canonical projection $p\colon \Pro^1 \to \overline{T}$ is a
motivic weak equivalence, because $\Aff^1$ is contractible.
It induces a Quillen equivalence
\[ \xymatrix{ \MSS(S)=\MSS_{\Pro^1}(S) \ar@<0.7ex>[r]^-{p_\sharp}  & \MSS_{\overline{T}}(S)\ar@<0.7ex>[l]^-{p^\ast}} \]
when equipped with model structures as described in \cite{J} (see \cite[A.6.9]{BGL}). The right adjoint
$p^\ast$ is very simple: it sends a symmetric $\overline{T}$-spectrum $E$ to the symmetric $\Pro^1$-spectrum
having terms $\bigl(p^\ast(E)\bigr)_n =E_n$ and structure maps
\[ \xymatrix@C=5em{E_n\wedge \Pro^1 \ar[r]^-{E_n\wedge p} & E\wedge \overline{T} \ar[r]^-{\mathrm{structure\ map}}
  & E_{n+1}}. \]
In particular $\mathrm{MGL} :=p^\ast \mathbb{MGL}$ is a
symmetric $\Pro^1$-spectrum by just changing the structure maps.
Since $p^\ast$ is a lax symmetric monoidal functor, $\mathrm{MGL}$ is a
commutative monoid in a canonical way.
Finally, the identity
is a left Quillen equivalence from the model category $\MSS^\cm(S)$ used in
\cite{BGL} to Jardine's model structure by the proof of \cite[A.6.4]{BGL}.
Let
$\gamma\colon \mathrm{Ho}(\MSS^\cm(S)) \to  \SH(S)$ denote the
equivalence obtained by regarding a symmetric $\Pro^1$-spectrum just as a
$\Pro^1$-spectrum.
%Let
%$\bar T$
%be a Nisnevich sheaf associated with the presheaf
%$X \mapsto \Pro^1(X)/(\Pro^1 - \{0\})(X)$
%on the Nisnevich site
%$\Sm/k$.
%The $\bar T$ regarded as a presheaf is a motivic space in the
%%sence of
%\cite[Defn.3.1.1]{BGL}.
%Note that the open inclusion
%$\Aff^1 = \Pro^1 - \{\infty\} \hra \Pro^1$
%induces a space isomorphism
%$T = \bar T$,
%since both are defined as Nisnevich sheaves.
%Using this space isomorphism consider
%$\mathbb{MGL}$
%as a symmetric ring
%$\bar T$-spectrum.
%Finally using the
%motivic weak equivalence
%$\Pro^1/\infty \to \bar T$
%we will consider
%$\mathbb{MGL}$
%as a symmetric ring
%$\Pro^1$-spectrum
%(see
%\cite[Defn.3.1.6]{BGL} for the definition of motivic weak equivalence).

%Let
%$\bar T$
%be a Nisnevich sheaf associated with the presheaf
%$X \mapsto \Pro^1(X)/(\Pro^1 - \{0\})(X)$
%on the Nisnevich site
%$\Sm/S$.
%The $\bar T$ regarded as a presheaf is a motivic space in the
%sence of
%\cite[Defn.3.1.1]{BGL}.
%Note that the open inclusion
%$\Aff^1 = \Pro^1 - \{\infty\} \hra \Pro^1$
%induces a space isomorphism
%$T = \bar T$,
%since both are defined as Nisnevich sheaves.

\begin{defn}
\label{textMGL}
Let
$(\mathrm{MGL},\mu_\mathrm{MGL},e_\mathrm{MGL})$
denote the
commutative $\Pro^1$-ring spectrum which is the image $\gamma(\mathrm{MGL})$
of the commutative symmetric $\Pro^1$-ring spectrum
$\mathrm{MGL}$ in the motivic stable homotopy category $\SH(S)$.
%We will regard
%$\mathrm{MGL}$
%just as a
%$\Pro^1$-spectrum
%using the forgetful functor from symmetric
%$\Pro^1$-spectra
%to
%%$\Pro^1$-spectra.
%The symmetric ring structure on
%$\mathbb {MGL}$
%gives a
%commutative monoid structure
%$(\mathrm{MGL}, \mu_{\mathrm{MGL}}, e_{\mathrm{MGL}})$
%on
%$\mathrm{MGL}$
%in the stable motivic homotopy category
%$\mathrm{SH(k)}$.
\end{defn}

\subsection{Cohomology of the algebraic cobordism spectrum}
Let $E$ be an oriented commutative $\Pro^1$-ring spectrum and
let $S=\Spec(k)$ for a field $k$. We will compute
$E^{\ast,\ast}(\mathrm{MGL})$
and
$E^{\ast,\ast}(\mathrm{MGL} \wedge \mathrm{MGL})$
in this short section.

By
\cite[Cor.~2.1.4]{BGL},
the group
$E^{\ast,\ast}(\mathrm{MGL})$
fits into the short exact sequence
\[
0 \to {\varprojlim}^{1}E^{\ast+2i-1,\ast+i}(\Th(\mathcal T(i)))
\to E^{\ast,\ast}(\mathrm{MGL})
\to \varprojlim E^{\ast+2i,\ast+i}(\Th(\mathcal T(i))) \to 0
\]
where the connecting maps in the tower are given by the top
line of the commutative diagram
\[ \xymatrix{
E^{\ast+2i-1,\ast+i}(\Th(i)) &
E^{\ast+2i+1,\ast+i+1}(\Th(i) \wedge \Pro^1) \ar[l]_-{\Sigma^{-1}_{\Pro^1}} &
E^{\ast+2i+1,\ast+i+1}(\Th(i+1)) \ar[l] \\
E^{\ast,\ast}(\Gr(i)) \ar[u]^-{-\cup th(\mathcal{T}(i))} &
E^{\ast,\ast}(\Gr(i)) \ar[l]_-{\id}
\ar[u]_-{\epsilon^\ast \circ(-\cup th(\mathcal{T}(i)\oplus \mathbf{1}))} &
E^{\ast,\ast}(\Gr(i+1)) \ar[l]_-{\mathrm{inc}_i^\ast} \ar[u]_-{-\cup th(\mathcal{T}(i+1))}}\]
Here
$\epsilon\colon \Th(V) \wedge \Pro^1 \to Th(V\oplus \mathbf{1})$
is the canonical map described in Corollary
\ref{ThomAndSuspension}.
The pull-backs
$\mathrm{inc}^\ast_i$
are all surjective by Theorem
\ref{ThomIsomorphism}.
So we proved the following
\begin{clm}
\label{CohomologyOfMGL}
The canonical map
\[
E^{\ast,\ast}(\mathrm{MGL}) \to \varprojlim E^{\ast+2i,\ast+i}(\Th(\mathcal T(i)))=
E^{\ast,\ast}(k)[[c_1,c_2,c_3,\dots]]
\]
is an isomorphism of two-sided
$E^{\ast,\ast}(k)$-modules.
\end{clm}

Now compute
$E^{\ast,\ast}(\mathrm{MGL} \wedge \mathrm{MGL})$.
By 
\cite[Cor.~2.1.5]{BGL}
the group
$E^{\ast,\ast}(\mathrm{MGL} \wedge \mathrm{MGL})$
fits into the short exact sequence
$$
0 \to {\varprojlim}^{1}E^{\ast+4i-1,\ast+2i}(\Th(\mathcal T(i)) \wedge \Th(\mathcal T(i)))
\to E^{\ast,\ast}(\mathrm{MGL} \wedge \mathrm{MGL})
$$
$$
\to \varprojlim E^{\ast+4i,\ast+2i}(Th(\mathcal T(i)) \wedge Th(\mathcal T(i))) \to 0.
$$
Note that since
$\Th(\mathcal T(i)) \wedge \Th(\mathcal T(i))\cong \Th(\mathcal T(i) \times \mathcal T(i))$,
there is a Thom isomorphism
$E^{\ast+4i-1,\ast+2i}(\Th(\mathcal T(i) \times \mathcal T(i))) \cong E^{\ast-1,\ast}(\Gr(i) \times \Gr(i))$
by Theorem
\ref{ThomIsomorphism}.
The
${\varprojlim}^{1}$-group is trivial
because the connecting maps coincide with the pull-back maps
$$
E^{*-1,*}(\Gr(i+1) \times \Gr(i+1)) \to E^{*-1,*}(\Gr(i) \times \Gr(i))
$$
and these are surjective by Theorem
\ref{CohomologyOfGrGr}.
This implies the following

\begin{clm}
\label{CohomologyOfMGLMGL}
The canonical map
$$
E^{\ast,\ast}(\mathrm{MGL} \wedge \mathrm{MGL}) \to
\varprojlim E^{\ast+2i,\ast+i}(\Th(\mathcal T(i)) \wedge \Th(\mathcal T(i)))=
$$
$$
E^{\ast,\ast}(k)[[c^{\prime}_1, c^{\prime\prime}_1, c^{\prime}_2, c^{\prime\prime}_2, \dots]]
$$
is an isomorphism of two-sided
$E^{\ast,\ast}(k)$-modules. Here
$c^{\prime}_i$
is the $i$-th Chern class coming from the first
factor of
$\mathrm{Gr} \times \mathrm{Gr}$
and
$c^{\prime\prime}_i$
is the $i$-th Chern class coming from the second factor.
\end{clm}

\subsection{A universality theorem for the algebraic cobordism spectrum}

The complex cobordism spectrum, equipped with its natural orientation,
is a universal oriented ring cohomology theory by Quillen's universality
theorem \cite{Quillen:mu}. In this section
we prove a motivic version of Quillen's universality theorem.
The statement is contained already in \cite{Vez}.
Recall that the $\Pro^1$-ring spectrum $\mathrm{MGL}$
carries a canonical orientation $th^{\mathrm{MGL}}$ as defined
in \ref{OrientationsOfMGLandK}.
%I THINK THAT properly speaking
%EVEN A MONOIDAL STRUCTURE ON MGL
%WAS NOT CONSTRUCTED in
%\cite{Vez}.
%We will write
%$\mathrm{MGL}$
%for
%$\mathbb{MGL}$
%%regarded as a symmetric ring
%$\Pro^1$-spectrum
%(see
%\ref{textMGL}).
%The symmetric ring structure on
%$\mathrm{MGL}$
%defines a commutative monoid structure
%$(\mathrm{MGL}, \mu_{\mathrm{MGL}}, e_{\mathrm{MGL}})$
%on
%$\mathrm{MGL}$
%regarded as an ordinary
%$\Pro^1$-spectrum.
%Let
%$(E, \mu_E, e_E)$
%be one more commutative ring $\Pro^1$-spectrum.
%\begin{defn}
%\label{OrientationOfMGL}
%Let
It is the canonical map
$th^{\mathrm{MGL}}\colon \Sigma^{\infty}_{\Pro^1}(Th(\mathcal O(-1)))(-1) \to \mathrm{MGL}$
of
$\Pro^1$-spectra. %Set
%$th^{\mathrm{MGL}} =u_1 \in \mathrm{MGL}^{2,1}(Th(\mathcal O(-1)))$.
%Clearly
%$th|_{Th(1)}= \Sigma_{\Pro^1}(1)$
%in
%$\mathrm{MGL}^{2,1}(Th(1))$. So the
%$th^{\mathrm{MGL}}$ is an orientation of
%$\mathrm{MGL}$
%(see Defn.
%\ref{OrientationViaThom}).
%\end{defn}

\begin{thm}[Universality Theorem]
\label{UniversalityThm}
Let $E$ be a commutative $\Pro^1$-ring spectrum and
let $S=\Spec(k)$ for a field $k$.
The assignment
$\varphi \mapsto \varphi(th^{\mathrm{MGL}}) \in E^{2,1}(\Th(\mathcal T(1)))$
identifies the set of monoid homomorphisms
\begin{eqnarray}
\label{OrientationMap}
\varphi\colon \mathrm{MGL} \to E
\end{eqnarray}
in the motivic stable homotopy category
$\mathrm{SH}^{\cm}(S)$
with the set of orientations of $E$.
The inverse bijection sends an orientation
$th \in E^{2,1}(\Th(\mathcal T(1)))$
to the unique morphism
\[\varphi \in E^{0,0}(\mathrm{MGL})=\Hom_{\SH(S)}(\mathrm{MGL},E)\]
such that
$u_i^\ast(\varphi)= th(\mathcal T(i)) \in E^{2i,i}(\Th(\mathcal T(i)))$,
where
$th(\mathcal T(i))$
is given by
(\ref{ThomClass})
and  \\
$u_i\colon \Sigma^{\infty}_{\Pro^1}(\Th(\mathcal T(i)))(-i) \to \mathrm{MGL}$
is the canonical map of $\Pro^1$-spectra.
\end{thm}

\begin{proof}
Let $\varphi\colon \mathrm{MGL}\to E$ be a homomorphism of monoids
in $\SH(S)$. The class
$th:=\varphi(th^{\mathrm{MGL}})$
is an orientation of $E$, because
\[ \varphi(th)|_{Th(\bf 1)}= \varphi(th|_{Th(\bf 1)})=
\varphi(\Sigma_{\Pro^1}(1))= \Sigma_{\Pro^1}(\varphi(1))= \Sigma_{\Pro^1}(1).\]
Now suppose
$th^E \in E^{2i,i}(\Th(\mathcal O(-1)))$
is an orientation of $E$. We will construct
a monoid homomorphism
$\varphi\colon \mathrm{MGL} \to E$
in $\SH(S)$ such that
$u_i^*(\varphi)= th(\mathcal T(i))$ and prove its uniqueness.
To do so recall that by Claim
\ref{CohomologyOfMGL}
the canonical map
$E^{\ast,\ast}(\mathrm{MGL}) \to \varprojlim E^{\ast+2i,\ast+i}(\Th(\mathcal T(i)))$
is an isomorphism.
The family of elements
$th(\mathcal T(i))$
is an element in the
$\varprojlim$-group, thus there is a unique
element $\varphi \in E^{0,0}(\mathrm{MGL})$
with
$u_i^*(\varphi)= th(\mathcal T(i))$.

We claim that $\varphi$ is a monoid homomorphism.
To check that it respects the multiplicative
structure, consider the diagram
\[ \xymatrix@C=6em{
\Sigma^{\infty}_{\Pro^1}(\Th(\mathcal T(i)))(-i) \wedge \Sigma^{\infty}_{\Pro^1}(\Th(\mathcal T(j)))(-j)
\ar[r]^-{\Sigma^{\infty}_{\Pro^1}(\mu_{i,j})(-i-j)} \ar[d]_-{u_i\wedge u_j} &
\Sigma^{\infty}_{\Pro^1}(\Th(\mathcal T(i+j)))(-i-j) \ar[d]^-{u_{i+j}} \\
\mathrm{MGL} \wedge \mathrm{MGL}  \ar[r]^-{\mu_{\mathrm{MGL}}} \ar[d]_-{\varphi \wedge \varphi}
& \mathrm{MGL}  \ar[d]^-{\varphi}  \\
E \wedge E       \ar[r]^-{\mu_E} &     E.} \]
Its enveloping square commutes in
$\mathrm{SH}(S)$
by the chain of relations
\begin{eqnarray*}
\varphi \circ u_{i+j} \circ \Sigma^{\infty}_{\Pro^1}(\mu_{i,j})(-i-j) & = &
\mu^\ast_{i,j}(th(\mathcal T(i+j)))=
th(\mu^\ast_{i,j}(\mathcal T(i+j)))=
th(\mathcal T(i) \times \mathcal T(j)) \\ & = &
th(\mathcal T(i)) \times th(\mathcal T(j))=
\mu_E(th(\mathcal T(i)) \wedge th(\mathcal T(j))) \\ & = &
\mu_E \circ ((\varphi \circ u_i) \wedge (\varphi \circ u_j)).
\end{eqnarray*}
The canonical map
$E^{\ast,\ast}(\mathrm{MGL} \wedge \mathrm{MGL})
\to
\varprojlim E^{\ast+4i,\ast+2i}(Th(\mathcal T(i)) \wedge Th(\mathcal T(i)))$
is an isomorphism by Claim
\ref{CohomologyOfMGLMGL}.
Now the equality
$$
\varphi \circ u_{i+i} \circ \Sigma^{\infty}_{\Pro^1}(\mu_{i,i})(-2i)=
\mu_E \circ ((\varphi \circ u_i) \wedge (\varphi \circ u_i))
$$
shows that
$\mu_E \circ (\varphi \wedge \varphi) = \varphi \circ \mu_{\mathrm{MGL}}$
in
$\mathrm{SH}(k)$.

To prove the Theorem it remains to check that the two assignments
described in the Theorem are inverse to each other.
An orientation
$th \in E^{2,1}(\Th(\mathcal O(-1)))$
induces a morphism
$\varphi$
such that for each $i$ one has
$\varphi \circ u_i = th(\mathcal T_i)$.
And the new orientation
$th^{\prime}:= \varphi(th^{\mathrm{MGL}})$
coincides with the original one, due to the chain of relations
$$
th^{\prime} = \varphi(th^{\mathrm{MGL}})= \varphi(u_1)= \varphi \circ u_1 =
th(\mathcal T(1))=th(\mathcal O(-1))=th.
$$

On the other hand
a monoid homomorphism
$\varphi$
defines an orientation
$th:= \varphi(th^{\mathrm{MGL}})$
of $E$. The  monoid homomorphism
$\varphi^{\prime}$ we obtain then
satisfies
$u^*_i(\varphi^{\prime})=th(\mathcal T(i))$
for every $i\geq 0$.
To check that
$\varphi^{\prime}=\varphi$,
recall that
$\mathrm{MGL}$
is oriented, so we may use
Claim
\ref{CohomologyOfMGL}
with
$E= \mathrm{MGL}$
to deduce an isomorphism
$$
\mathrm{MGL}^{\ast,\ast}(\mathrm{MGL})
\to \varprojlim \mathrm{MGL}^{\ast+2i,\ast+i}(\Th(\mathcal T(i))).
$$
This isomorphism shows that  the identity
$\varphi^{\prime}=\varphi$
will follow from the identities
$u^*_i(\varphi^{\prime})= u^*_i(\varphi)$
for every $i\geq 0$.
Since
$u^*_i(\varphi^{\prime})=th(\mathcal T_i)$
it remains to check the relation
$u^*_i(\varphi)= th(\mathcal T(i))$.
It follows from the
\begin{clm}
\label{UandThom}
There is an equality $u_i= th^{\mathrm{MGL}}(\mathcal T(i) \in \mathrm{MGL}^{2i,i}(\Th(\mathcal T(i)))$.
\end{clm}
In fact,
$u^\ast_i(\varphi)= \varphi \circ u_i = \varphi(u_i)=
\varphi(th^{\mathrm{MGL}}(\mathcal T(i)))=
th(\mathcal T(i))$.
The last equality in this chain of relations holds,
because
$\varphi$
is a monoid homomorphism sending
$th^{\mathrm{MGL}}$
to
$th$. It remains to prove the Claim.
We will do this in the case $i=2$. The general case can be proved
similarly. The commutative diagram
\[ \xymatrix@C=5em{
\Sigma^{\infty}_{\Pro^1}\Th(\mathcal T(1))(-1) \wedge
\Sigma^{\infty}_{\Pro^1}\Th(\mathcal T(1))(-1)
\ar[r]^-{\Sigma^{\infty}_{\Pro^1}(\mu_{1,1})(-2)} \ar[d]_-{u_1\wedge u_1} &
\Sigma^{\infty}_{\Pro^1}\Th(\mathcal T(2))(-2)   \ar[d]^-{u_2}   \\
\mathrm{MGL} \wedge \mathrm{MGL}  \ar[r]^-{\mu_{\mathrm{MGL}}} &\mathrm{MGL}}\]
in
$\mathrm{SH}(k)$
implies that
$$
\mu^\ast_{1,1}(u_2)= u_1 \times u_1
\in \mathrm{MGL}^{4,2}(\Th(\mathcal T(1)) \wedge \Th(\mathcal T(1)))=
\mathrm{MGL}^{4,2}(\Th(\mathcal T(1) \times \mathcal T(1))).
$$
The equalities
\begin{eqnarray*}
\mu^\ast_{1,1}(th^{\mathrm{MGL}}(\mathcal T(2))) & = &
th^{\mathrm{MGL}}(\mu^\ast_{1,1}(\mathcal T(2)))=
th^{\mathrm{MGL}}(\mathcal T(1) \times \mathcal T(1)) \\ & =&
th^{\mathrm{MGL}}(\mathcal T(1)) \times th^{\mathrm{MGL}}(\mathcal T(1))
\end{eqnarray*}
imply that it remains to prove the injectivity of the map
$\mu^\ast_{1,1}$. Consider the commutative diagram
\[
\xymatrix{
\mathrm{MGL}^{\ast,\ast}(\Th(\mathcal T(1) \times \mathcal T(1)))&
\mathrm{MGL}^{\ast,\ast}(\Th(\mathcal T(2)))  \ar[l]_-{\mu^\ast_{1,1}}   \\
\mathrm{MGL}^{\ast,\ast}(\Gr(1) \times \Gr(1)) \ar[u]^-{\mathrm{Thom}}_-\cong &   \mathrm{MGL}^{\ast,\ast}(\Gr(2))
\ar[l]_-{\nu^\ast_{1,1}} \ar[u]_-{\mathrm{Thom}}^\cong}\]
where the vertical arrows are the Thom isomorphisms from Theorem
\ref{ThomIsomorphism}
and
$\nu_{1,1}: \Gr(1) \times \Gr(1) \hra \Gr(2)$
is the embedding described by equation~(\ref{eq:111}).
%the very beginning of Section
%\ref{SpectrumMGL}.
For an oriented commutative $\Pro^1$-ring spectrum
$(E,th)$
one has
$E^{\ast,\ast}(\Gr(2))=E^{\ast,\ast}(k)[[c_1,c_2]]$
(the formal power series on $c_1$, $c_2$)
by Theorem
\ref{CohomologyOfGr}.
From the other hand
$$
E^{\ast,\ast}(\Gr(1) \times \Gr(1))= E^{\ast,\ast}(k)[[t_1,t_2]]
$$
(the formal power series on $t_1$, $t_2$)
by Theorem
\ref{CohomologyOfGrGr}
and the map
$\nu^\ast_{1,1}$
takes
$c_1$ to $t_1+t_2$
and
$c_2$ to
$t_1t_2$.
Whence
$\nu^\ast_{1,1}$
is injective.
The proofs of the Claim and of the Theorem are completed.
\end{proof}

\end{document}